\documentclass[11pt]{amsart}
\usepackage{amssymb,mathrsfs,graphicx,enumerate}

\title[  $L^2$-contraction for viscous shocks ]{$L^2$-contraction for shock waves of scalar viscous conservation laws}

\author[Kang]{Moon-Jin Kang}
\address[Moon-Jin Kang]{\newline Department of Mathematics, \newline The University of Texas at Austin, Austin, TX 78712, USA}
\email{moonjinkang@math.utexas.edu}

\author[Vasseur]{Alexis F. Vasseur}
\address[Alexis F. Vasseur]{\newline Department of Mathematics, \newline The University of Texas at Austin, Austin, TX 78712, USA}
\email{vasseur@math.utexas.edu}

\newtheorem{theorem}{Theorem}[section]
\newtheorem{lemma}{Lemma}[section]

\newtheorem{remark}{Remark}[section]

\newcommand{\bbr}{\mathbb R}

\newcommand{\R}{\mathbb R}

\numberwithin{figure}{section}
\newcommand{\beq}{\begin{equation}}
\newcommand{\eeq}{\end{equation}}
\newcommand{\bsp}{\begin{split}}
\newcommand{\esp}{\end{split}}

\newcommand{\sgn}{{\text{\rm sgn}}}

\newcommand{\RR}{{\mathbb R}}

\def\eps{\varepsilon }

\newcommand\adots{\mathinner{\mkern2mu\raise1pt\hbox{.}
\mkern3mu\raise4pt\hbox{.}\mkern1mu\raise7pt\hbox{.}}}

\def\charf {\mbox{{\text 1}\kern-.30em {\text l}}}

\begin{document}

\date{\today}

\subjclass[2010]{35L65, 35L67, 35B35, 35B40} \keywords{viscous conservation laws, shock wave, stability, contraction, relative entropy}

\thanks{\textbf{Acknowledgment.} M.-J. Kang was supported by Basic Science Research Program through the National Research Foundation of Korea funded by the Ministry of Education, Science and Technology (NRF-2013R1A6A3A03020506). A. F. Vasseur was partially supported by the NSF Grant DMS 1209420. The authors thank Prof. Denis Serre for valuable comments and suggesting using the entropy method for  the viscous conservation laws.
}

\begin{abstract} 
We consider the $L^2$-contraction up to a shift for viscous shocks of  scalar viscous conservation laws with strictly convex fluxes in one space dimension. In the case of a flux which is a small perturbation of the quadratic Burgers flux,  we show that any viscous shock induces  a contraction in $L^2$, up to a shift. That is,  the $L^2$ norm of the difference of any solution of the viscous conservation law, with an appropriate shift of the  shock wave, does not increase in time.   If, in addition, the difference between the initial value of the solution and the shock wave is also bounded in $L^1$, the $L^2$ norm of the  difference  converges  at the optimal rate $t^{-1/4}$. Both results do not involve any smallness condition on the initial value, nor on the size of the shock. In this context of small perturbations of the quadratic Burgers flux, the result improves  the Choi and Vasseur's result in \cite{C-V}. However, we show that the $L^2$-contraction up to a shift does not hold for every convex flux. We construct a smooth strictly  convex flux, for which the $L^2$-contraction does not hold any more even along any Lipschitz shift. 
\end{abstract}
\maketitle \centerline{\date}


\section{Introduction and main results}
\setcounter{equation}{0}
This paper is devoted to the study of $L^2$-contraction properties, up to a shift,
 for viscous shock waves of  scalar viscous conservation laws with smooth strictly convex fluxes $A$ in one space dimension: 
\begin{align}
\begin{aligned} \label{main}
&\partial_t U + \partial_x A(U) = \partial^2_{xx} U, \quad t>0,~x\in \bbr,\\
&U(0,x) = U_0(x).
\end{aligned}
\end{align}
For any smooth strictly convex flux $A$, and any  $u_{-}, u_{+}\in \RR$  with $u_{-}>u_{+}$, there exists a smooth   function   $S_1$ defined on $\RR$, and $\sigma\in \RR$, such that  $S_1(x-\sigma t)$ is a traveling wave solution of  Equation \eqref{main},   connecting  $u_-$ at $-\infty$ to $u_+$ at $+\infty$. The function $S_1$ satisfies  
\begin{align}
\begin{aligned}\label{layer-0} 
& -\sigma S_1^{\prime}(\xi) + A(S_1)^{\prime}(\xi) = S_1^{\prime\prime}(\xi),\\
& \lim_{\xi\to \pm\infty} S_1 = u_{\pm},\quad \lim_{\xi\to \pm\infty} S_1^{\prime} = 0,
\end{aligned}
\end{align} 
where $\sigma$ is the speed of the shock determined by the Rankine-Hugoniot condition:
\beq\label{RH-con}
\sigma=\frac{A(u_{+}) -  A(u_{-})}{u_{+}-u_{-}}.
\eeq
Integrating \eqref{layer-0}, we find
\begin{align}
\begin{aligned}\label{layer} 
 -\sigma(S_1 - u_{\pm})+A(S_1) -  A(u_{\pm}) = S_1^{\prime}, \quad \lim_{\xi\to \pm\infty} S_1 = u_{\pm}.
\end{aligned}
\end{align}

There have been extensive studies on the stability of shock profiles of viscous conservation laws. When initial data $U_0$ is a small perturbation from the viscous shock $S_1$, the stability estimates have been shown in various way, such as the maximum principle, Evans function theory and the weighted norm approach based on the semigroup framework. This kind of results in the scalar case have been obtained by Goodman \cite{G-2}, Hopf \cite{H}, Howard \cite{Howard}, Nishihara \cite{N}, but also in the system case by Liu \cite{L-1, L-2} and Zumbrun \cite{L-Z-2}, Szepessy and Xin \cite{S-X} (see also \cite{G-1}). On the other hand, Freist$\ddot{\mbox{u}}$hler and Serre \cite {F-S} have shown the $L^1$-stability of viscous shock waves,  without smallness condition, by combining energy estimates, a lap-number argument and a specific geometric observation on attractor of steady states. Moreover, their stability result still holds for any $L^p$ space, $1\le p\le \infty$. This result was improved by Kenig and Merle \cite{K-M}, with the uniform convergence to the viscous shock with respect to initial datas. The contraction property of viscous scalar conservation laws with respect to  Wasserstein distances, was studied by  Bolley, Brenier, and  Loeper  in \cite{B-B-L},  and Carrillo, Francesco and Lattanzio in \cite{C-F-L}.\\
   
 In this article, we use the relative entropy method to study contraction properties in $L^2$ for viscous shocks to scalar viscous conservation laws. This work follows a program initiated in   \cite{L,L-V,S-V,V, V-1} concerning the relative entropy method for the study on the stability of inviscid shocks for the scalar or system of conservation laws verifying a certain entropy condition. The relative entropy method has been used as an important tool in the study of asymptotic limits to conservation laws as well. For incompressible limits, see Bardos, Golse, Levermore  \cite{B-G-L-1,B-G-L-2}, Lions and Masmoudi \cite{L-M}, Saint Raymond \cite{S}. For the compressible limit, see Tzavaras \cite{T} in the context of relaxation and \cite{B-T-V,B-V,K-V,M-V,Y} in the context of hydrodynamical limits.\\

Our first result is on the $L^2$-contraction up to a shift for viscous shocks of \eqref{main} with a strictly convex flux $A$ which has a perturbed form of quadratic function as
\beq\label{quad}
A(x)=ax^2 + g(x), \quad  a>0,
\eeq
where $g$ is a $C^2$-function satisfying $\|g^{\prime\prime}\|_{L^{\infty}(\bbr)} <\frac{2}{11}a$. 
The following result shows  also a rate of  convergence toward the shock waves  as  $t^{-1/4}$, as long as the initial perturbation $U_0-S_1$ is also  bounded in $L^1$. Notice that the decay rate $t^{-1/4}$  is the same rate as the heat equation. Moreover, our result does not need any  assumption on the spatial decay of the  initial data, in contrast with   previous works (see for example \cite{H, N}).
\begin{theorem}\label{main thm} 
Assume the flux $A$ as in \eqref{quad}. For any given $u_{-}>u_{+}$, let $S_1$ be the associated viscous layer of \eqref{main} with endpoints $u_{-}$ and $u_{+}$. Then, for any solution $U$ to \eqref{main} with  initial data $U_0$ satisfying $U_0-S_1\in L^{2}(\bbr)$, the following $L^2$-contraction holds:
\begin{align}
\begin{aligned}\label{contract}
\|U(t,\cdot+X(t))-S_1\|_{L^{2}(\bbr)} \le \|U_0-S_1\|_{L^{2}(\bbr)},\quad t>0,
\end{aligned}
\end{align}
for the shift $X(t)$ satisfying
\begin{align}
\begin{aligned}\label{curve}
& \dot{X}(t) = \sigma -\frac{2a+\|g^{\prime\prime}\|_{L^{\infty}(\bbr)}}{2(u_{-}- u_{+})}\int_{-\infty}^{\infty} (U(t,x+X(t))- S_1(x)) S_1^{\prime}(x) dx,\\
& X(0)=0.
\end{aligned}
\end{align}
Furthermore, if $U_0-S_1\in L^{1}\cap L^{2}(\bbr)$, we have the following estimate for all $t>0$,
\beq\label{convergence}
\|U(t,\cdot+X(t))-S_1\|_{L^{2}(\bbr)} \le \frac{ C_0\|U_0 - S_1\|_{L^2(\bbr)}}{ C_0+t^{1/4}\|U_0 - S_1\|_{L^2(\bbr)} },
\eeq
where $C_0:= C(1+  \|U_0-S_1\|_{L^{1}(\bbr)}+ \|U_0-S_1\|_{L^{2}(\bbr)}^2)$ and $C$ is a positive constant only depending on the end points $u_-$, $u_+$ and the flux $A$. 
\end{theorem}

\begin{remark}
The existence and uniqueness of the curve $X$ are guaranteed by the Cauchy-Lipschitz theorem. Moreover, $X$ is Lipschitz. 
Indeed, since $A^{\prime\prime}>0$, it follows from \eqref{layer} that $S_1$ satisfies $u_+<S_1<u_-$ and
\begin{align}
\begin{aligned}\label{s-1-1} 
S_1^{\prime} &= - \sigma (S_1 -u_{\pm}) + A(S_1) -A(u_{\pm})\\
& = (u_{\pm} - S_1) \Big(  \frac{A(u_{-}) -A(u_{+})}{u_{-}-u_{+}} -  \frac{A(u_{\pm}) -A(S_1)}{u_{\pm}-S_1} \Big) <0.
\end{aligned}
\end{align}
In particular, since
\[
\frac{A(u_{\pm}) -A(S_1)}{u_{\pm}-S_1}\rightarrow A^{\prime} (u_{\pm}) \quad \mbox{as} ~S_1\rightarrow u_{\pm},
\]
for some positive constants $c_{\pm}$, we have
\begin{align*}
\begin{aligned}
(u_{\pm} - S_1) \Big(  \frac{A(u_{-}) -A(u_{+})}{u_{-}-u_{+}} -  \frac{A(u_{\pm}) -A(S_1)}{u_{\pm}-S_1} \Big) \sim \pm c_{\pm}(u_{\pm} - S_1) \quad \mbox{as} ~S_1\rightarrow u_{\pm},
\end{aligned}
\end{align*}
which implies
\begin{align*}
\begin{aligned}
|S_1(\xi) -u_{\pm}| \sim  \exp(-c_{\pm} |\xi|)\quad\mbox{as} ~\xi\rightarrow \pm \infty.
\end{aligned}
\end{align*}
This yields $S_1^{\prime}\in L^{2}(\bbr)$, therefore, using \eqref{contract}, we have
\begin{align*}
\begin{aligned}
|\dot{X}(t)-\sigma| &\le  \frac{2a+\eps}{2 (u_{-}- u_{+})}  \|U(t,\cdot+X(t))-S_1\|_{L^{2}(\bbr)} \|S_1^{\prime}\|_{L^{2}(\bbr)}\\
&\le C \|U(t,\cdot+X(t))-S_1\|_{L^{2}(\bbr)}\\
& \le C  \|U_0-S_1\|_{L^{2}(\bbr)}.
\end{aligned}
\end{align*}
\end{remark}
\vspace{0.5cm}

As a second result, we  construct a strictly convex flux $A$, for which a viscous shock of \eqref{main} does not induce a $L^2$-contraction up to a shift. This is stated in the following theorem.
\begin{theorem}\label{thm-example} 
For any given $u_{-}>u_{+}$, there is a smooth strictly convex flux $A$ and smooth initial data $U_0$ with $U_0-S_1\in L^{2}(\bbr)$ such that for any Lipschitz shift $X$, there exists $T^*>0$, such that the solution $U$ to \eqref{main} with $A$ and $U_0$ satisfies
\begin{align*}
\begin{aligned}
\|U(t,\cdot+X(t))-S_1\|_{L^{2}(\bbr)} > \|U_0-S_1\|_{L^{2}(\bbr)},\quad 0\le t <T^*.
\end{aligned}
\end{align*}
\end{theorem}

\vspace{0.5cm}
As an  application of Theorem \ref{main thm}, the contraction \eqref{contract} and decay estimate \eqref{convergence} can be applied to the study on the inviscid limit to the shock waves. In \cite{C-V}, Choi and Vasseur considered the following equation
\begin{align}
\begin{aligned} \label{inviscid main}
&\partial_t U^{\eps} + \partial_x A(U^{\eps}) = \eps\partial^2_{xx} U^{\eps}, \quad t>0,~x\in \bbr,\\
&U^{\eps}(0,x) = U_0(x).
\end{aligned}
\end{align}
They showed that the rate of convergence  in $L^2$ up to a shift,  to an inviscid shock, is of order 
$\sqrt{\eps}\log{(1/\eps)}$.
Let us denote
\beq\label{shock-0}
S_0(x)=\left\{ \begin{array}{ll}
          u_{-}& \mbox{if $ x < 0$},\\
        u_{+}& \mbox{if $x\ge 0$},\end{array} \right.
\eeq

Theorem \ref{main thm} improves the rate of convergence, and simplifies the assumptions  in their result. Indeed, as a third result,  we  show the following theorem.
\begin{theorem}\label{thm-improve} 
Under the same hypothesis of Theorem \ref{main thm},  the solution $U^{\eps}$ to \eqref{inviscid main} verifies
\beq\label{im-1}
\|U^{\eps}(t,\cdot) - S_0(\cdot-Y(t))\|_{L^2} \leq \|U_0 - S_0\|_{L^2} + C\sqrt \eps,\quad t>0,
\eeq
where the shift $Y$ is defined by $Y(t)=\eps X(t/\eps)$ from the shift $X$ defined in  \eqref{curve}.\\
Moreover, if 
\beq\label{small-1}
\int_{\bbr} \Big|U_0(x) - S_0(x)\Big|^2 dx+\int_{\bbr} \Big|U_0(x) - S_0(x)\Big| dx\le C\eps,
\eeq
 then we have
\beq\label{im-2}
\Big\|U^{\eps}(t,\cdot) - S_1\Big(\frac{\cdot-Y(t)}{\eps} \Big)\Big\|_{L^{2}}^2 \le \frac{C\eps^{3/2}}{\eps^{1/2}+ t^{1/2} }, \quad t>0.
\eeq
\end{theorem}


The rest of the paper is organized as follows. In Section 2, we present our framework and the relative entropy method. The Section 3 is devoted to the proof of Theorem \ref{main thm}. In Section 4, we prove Theorem \ref{thm-example} by constructing a specific flux function and initial data. In Section 5, we present the proof for the Theorem \ref{thm-improve}.  

\section{Preliminaries}
\setcounter{equation}{0}

\subsection{moving frame} 
For simplicity of the proof of the main results, we consider a moving framework along the drift Lipschitz curve $X$. More precisely, we employ a new function $V$ as follows:
\[
V(t,x):=U(t, x+X(t)),
\]
where $U$ is a solution to \eqref{main}. Then, we can easily check that $V$ verifies
\begin{align}
\begin{aligned} \label{V-eq}
&\partial_t V -\dot{X}(t) \partial_x V + \partial_x A(V) = \partial^2_{xx} V, \quad t>0,~x\in \bbr,\\
&V(0,x) = U_0(x).
\end{aligned}
\end{align}
 
\subsection{relative entropy method}
In this part, we present the $L^2$-framework as the following lemma, based on the relative entropy method.
\begin{lemma}\label{lem-entropy}
Let $S_1$ be a viscous shock given by \eqref{layer}. Then, the shock $S_1$ is a monotone function and $y=S_1(x)$ is an admissible change of variable. If we use $w$ defined by
\beq\label{w}
w(t, S_1(x)):= V(t,x) -  S_1(x),
\eeq
then the solution $V$ of \eqref{V-eq} satisfies 
\begin{align}
\begin{aligned}\label{main-energy}
\frac{d}{dt} \int_{-\infty}^{\infty} |V-S_1|^2 dx + D(t) =0,
\end{aligned}
\end{align}
where the dissipation $D(t)$ is given by 
\begin{align}
\begin{aligned} \label{dissipate}
D(t) &= 2 (\dot{X}(t) -\sigma) \int_{u_{+}}^{u_{-}} w dy -2  \int_{u_{+}}^{u_{-}} A(w+y|y) dy  \\
 &\quad -2  \int_{u_{+}}^{u_{-}} \Big( A(y) -A(u_{-}) -\sigma (y -u_{-}) \Big) |\partial_y w|^2 dy.
\end{aligned}
\end{align}
\end{lemma}

The remaining part of this section is devoted to the proof of Lemma \ref{lem-entropy}. Even though our framework is based on the $L^2$-norm, we here present the general case of the relative entropy $\eta(\cdot|\cdot)$ for a given entropy $\eta$. Then, we will focus on the quadratic entropy and explain why the choice of quadratic entropy is essential. Concerning the following relative entropy method, we refer to \cite{D}, \cite{L} and \cite{V-1}. \\
For any strictly convex entropy $\eta$ of $\eqref{main}$, we define the associated relative entropy function by
\[
\eta(u|v)=\eta(u)-\eta(v) -\eta^{\prime}(v) (u-v).
\]
Let $F(\cdot,\cdot)$ be the flux of the relative entropy defined by
\[
F(u,v) = G(u)-G(v) -\eta^{\prime} (v) (A(u)-A(v)),
\]
where $G$ is the entropy flux of $\eta$, i.e., $G^{\prime} = \eta^{\prime} A^{\prime}$.\\
We want to investigate the relative entropy between the solution $V$ of \eqref{V-eq} and the viscous shock $S_1$ defined in \eqref{layer}. We first notice that since $S_1$ does not depend on $t$,
\begin{align*}
\begin{aligned} 
\partial_t \eta(V|S_1) = (\eta^{\prime}(V) - \eta^{\prime}(S_1)) \partial_t V.
\end{aligned}
\end{align*}
We add the term concerning $S_1$ to the above equation by using \eqref{layer-0} and $\partial_{v}\eta(u|v) = -\eta^{\prime\prime} (v) (u-v)$, that is,
\begin{align*}
\begin{aligned} 
\partial_t \eta(V|S_1)  = (\eta^{\prime}(V) - \eta^{\prime}(S_1)) \partial_t V + \eta^{\prime\prime}(S_1) (V-S_1) (-\sigma S_1^{\prime} + A(S_1)^{\prime} - S_1^{\prime\prime}).
\end{aligned}
\end{align*} 
Then we use \eqref{V-eq} to get
\begin{align*}
\begin{aligned} 
\partial_t \eta(V|S_1) & =  (\eta^{\prime}(V) - \eta^{\prime}(S_1))  (\dot{X}(t)\partial_x V -A^{\prime}(V)\partial_x V+ \partial_{xx}V ) \\
& \qquad + \eta^{\prime\prime}(S_1) (V-S_1) (-\sigma S_1^{\prime} + A^{\prime}(S_1)S_1^{\prime} - S_1^{\prime\prime}).
\end{aligned}
\end{align*} 
If we use the relative flux defined by 
\[
A(u|v):=A(u)-A(v)-A^{\prime}(v)(u-v),
\]
and 
\begin{align*}
\begin{aligned} 
\partial_x F(V,S_1) &= \eta^{\prime}(V)A^{\prime}(V)\partial_x V - \eta^{\prime}(S_1)A^{\prime}(S_1)S_1^{\prime}\\
&\quad -\eta^{\prime\prime}(S_1)S_1^{\prime} (A(v)-A(S_1)) -\eta^{\prime} (S_1) (A^{\prime} (V)\partial_x V - A^{\prime}(S_1) S_1^{\prime}),
\end{aligned}
\end{align*} 
then we have
\begin{align*}
\begin{aligned} 
-\partial_x F(V,S) -\eta^{\prime\prime} (S_1) S_1^{\prime}A(V|S) = - (\eta^{\prime}(V) - \eta^{\prime}(S_1))A^{\prime}(V)\partial_x V + \eta^{\prime\prime}(S_1) (V-S_1)A^{\prime}(S_1) S_1^{\prime}. 
\end{aligned}
\end{align*} 
Thus we have
\begin{align}
\begin{aligned}\label{eta-est-0} 
\partial_t \eta(V|S_1)&=\dot{X}(t) (\eta^{\prime}(V) - \eta^{\prime}(S_1)) \partial_x V -\sigma \eta^{\prime\prime}(S_1) (V-S_1) S_1^{\prime}  -\partial_x F(V,S)\\
& \qquad   -\eta^{\prime\prime} (S_1) S_1^{\prime}A(V|S_1)+  (\eta^{\prime}(V) - \eta^{\prime}(S_1)) \partial_{xx}V -\eta^{\prime\prime}(S_1) (V-S_1) S_1^{\prime\prime}.
\end{aligned}
\end{align}
We now integrate \eqref{eta-est-0} in $x$ to get
\begin{align}
\begin{aligned}\label{eta-est} 
&\frac{d}{dt}  \int_{-\infty}^{\infty} \eta(V|S_1) dx\\
&\quad =\dot{X}(t) \int_{-\infty}^{\infty} (\eta^{\prime}(V) - \eta^{\prime}(S_1)) \partial_x V dx -\sigma  \int_{-\infty}^{\infty}  \eta^{\prime\prime}(S_1) (V-S_1) S_1^{\prime}  dx\\
& \qquad   - \int_{-\infty}^{\infty}\eta^{\prime\prime} (S_1) S_1^{\prime}A(V|S) dx+\int_{-\infty}^{\infty}\Big( (\eta^{\prime}(V) - \eta^{\prime}(S_1)) \partial_{xx}V -\eta^{\prime\prime}(S_1) (V-S_1) S_1^{\prime\prime} \Big) dx.
\end{aligned}
\end{align}
From now on, we only consider the quadratic entropy $\eta(u)=u^2$. This choice ensures that the parabolic term induces a positive dissipation. Moreover, since
\begin{align*}
\begin{aligned}
& 2\dot{X}(t) \int_{-\infty}^{\infty} (V - S_1) \partial_x V dx -2\sigma  \int_{-\infty}^{\infty}   (V-S_1) S_1^{\prime}  dx\\
&\qquad = \dot{X}(t) \int_{-\infty}^{\infty}  \partial_x(V-S_1)^2 dx+ 2 (\dot{X}(t) - \sigma) \int_{-\infty}^{\infty}(V-S_1)  S_1^{\prime}dx,
\end{aligned}
\end{align*}
we can reduce \eqref{eta-est} to 
\begin{align}
\begin{aligned} \label{main-energy-1}
\frac{d}{dt} \int_{-\infty}^{\infty} |V-S_1|^2 dx &= 2 (\dot{X}(t) - \sigma) \int_{-\infty}^{\infty}(V-S_1)  S_1^{\prime}dx - 2\int_{-\infty}^{\infty}A(V|S_1)S_1^{\prime} dx \\
&\quad - 2 \int_{-\infty}^{\infty}|\partial_x(V-S_1)|^2 dx\\
&=: -D(t),
\end{aligned}
\end{align}
where the $D(t)$ denotes the dissipation term.\\
We now use the change of variable $y=S_1(x)$, which is admissible thanks to \eqref{s-1-1}. Thus, if we define $w$ as the perturbation $V-S_1$ by
\[
w(t, S_1(x)):= V(t,x) -  S_1(x),
\]
then the dissipation $D(t)$ in \eqref{main-energy-1} becomes \eqref{dissipate}. This completes the proof of  Lemma \ref{lem-entropy}.

\section{Proof of Theorem \ref{main thm}} 
We first prove the contraction \eqref{contract} for any initial perturbation $U_0-S_1\in L^2$, then derive decay estimate \eqref{convergence}, for which we only need an additional assumption $U_0-S_1\in L^1$.
\subsection{Contraction for viscous shock}
In this part, we show the contraction by estimating the dissipation $D(t)$ to be nonnegative. We consider the perturbed quadratic flux $A(U)$ in the sense \eqref{quad}, i.e.,
\beq\label{aaa}
A(U) =a U^2 + g(U), \quad a>0,
\eeq
with any $C^2$-function $g$  satisfying $\|g^{\prime\prime}\|_{L^{\infty}(\bbr)} < \frac{2}{11}a$.\\
For the flux $A$ in \eqref{aaa}, the dissipation $D(t)$ in \eqref{dissipate} becomes
\begin{align*}
\begin{aligned} 
D(t) &= 2 (\dot{X}(t) -\sigma) \int_{u_{+}}^{u_{-}} w dy - 2\int_{u_{+}}^{u_{-}} \Big(  aw^2 + g(w+y)-g(y) -g^{\prime}(y) w \Big) dy \\
&\quad + 2 a\int_{u_{+}}^{u_{-}}  (u_{-}-y)(y-u_{+})  |\partial_y w|^2 dy\\
&\quad - 2 \int_{u_{+}}^{u_{-}} \underbrace{\Big( g(y) -g(u_{-}) -\frac{g(u_-) - g(u_+)}{u_- -u_+} (y -u_{-}) \Big)}_{J} |\partial_y w|^2 dy\\
&=:\sum_{k=1}^4\mathcal{I}_k.
\end{aligned}
\end{align*}
We want to show $D(t) > 0$ by using a shift function $X$ defined by \eqref{curve}. Then, by \eqref{w}, we have
\begin{align*}
\begin{aligned} 
\dot{X}(t)-\sigma &= -\frac{2a+\|g^{\prime\prime}\|_{L^{\infty}(\bbr)}}{ 2(u_{-}- u_{+})}\int_{-\infty}^{\infty} (V(t,x)- S_1(x)) S_1^{\prime}(x) dx\\
&=\frac{2a+\|g^{\prime\prime}\|_{L^{\infty}(\bbr)}}{ 2(u_{-}- u_{+})}\int_{u_{+}}^{u_{-}} w dy.
\end{aligned}
\end{align*}
We denote by $\bar{w}$  the mean of $w$, i.e., $\displaystyle \bar{w}(t):= \frac{1}{\alpha} \int_{u_{+}}^{u_{-}} w dy$, where $\alpha :=u_{-}-u_{+}$ is the shock strengh.\\
Thus we have
\beq\label{i_1}
 \mathcal{I}_1 =(2a+\|g^{\prime\prime}\|_{L^{\infty}})\alpha\bar{w}^2.
\eeq
For $ \mathcal{I}_2$, since
\begin{align*}
\begin{aligned} 
|g(w+y)-g(y)-g^{\prime}(y)w| \le \|g^{\prime\prime}\|_{L^{\infty}}\frac{ w^2}{2},
\end{aligned}
\end{align*}
we have
\beq\label{i_2}
 \mathcal{I}_2 \ge -(2a+\|g^{\prime\prime}\|_{L^{\infty}}) \int_{u_{+}}^{u_{-}} w^2 dy.
\eeq
We now combine \eqref{i_1} and \eqref{i_2} to get
\begin{align*}
\begin{aligned}
\mathcal{I}_1+\mathcal{I}_2 \ge (2a+\|g^{\prime\prime}\|_{L^{\infty}})\Big(\alpha\bar{w}^2 - \int_{u_{+}}^{u_{-}}w^2 dy\Big) =- (2a+\|g^{\prime\prime}\|_{L^{\infty}}) \int_{u_{+}}^{u_{-}} (w-\bar{w})^2 dy.
\end{aligned}
\end{align*}
For $\mathcal{I}_4$, we rewrite $J$ as
\[
J=\Big( \frac{g(y) -g(u_{-})}{y -u_{-}} -\frac{g(u_+) - g(u_-)}{u_+ -u_-}\Big) (y -u_{-}).
\]
Applying the Taylor theorem to $F(y):=\frac{g(y)-g(u_-)}{y-u_-}$ and then to $g$, we have 
\begin{align*}
\begin{aligned}
J= F'(y_*) (y-u_+) (y-u_-)=-\frac{1}{2}g''(y_{**})(y-u_+) (y-u_-),
\end{aligned}
\end{align*}
which yields
\beq\label{I_4}
|\mathcal{I}_4|\le 2\int_{u_{+}}^{u_{-}} |J| |\partial_y w|^2 dy \le \|g^{\prime\prime}\|_{L^{\infty}(\bbr)} \int_{u_{+}}^{u_{-}}  (u_{-}- y)(y-u_{+})|\partial_y w|^2 dy.
\eeq
Therefore, the dissipation $D(t)$ can be estimated as
\begin{align*}
\begin{aligned} 
D(t) &\geq- (2a+ \|g^{\prime\prime}\|_{L^{\infty}(\bbr)}) \int_{u_{+}}^{u_{-}} (w-\bar{w})^2 dy\\
&\qquad + (2 a- \|g^{\prime\prime}\|_{L^{\infty}(\bbr)})\int_{u_{+}}^{u_{-}}  (u_{-}-y)(y-u_{+})  |\partial_y w|^2 dy.
\end{aligned}
\end{align*}
To complete $D(t)\ge 0$, we use the weighted Poincar\'e type inequality in the following lemma. Indeed, applying Lemma \ref{lem-poincare} to the above estimate, we have
\begin{align}\label{d-2}
\begin{aligned} 
D(t) &\geq \lambda \int_{u_{+}}^{u_{-}}  (u_{-}-y)(y-u_{+})  |\partial_y w|^2 dy,
\end{aligned}
\end{align}
where 
\[
\lambda = 2a- 11 \|g^{\prime\prime}\|_{L^{\infty}(\bbr)}.
\]
Since $\|g^{\prime\prime}\|_{L^{\infty}(\bbr)}\le \frac{2}{11}a$, we have
\[
\frac{d}{dt} \int_{-\infty}^{\infty} |V-S_1|^2 dx = -D(t) \le 0,
\]
which implies the contraction \eqref{contract}. Hence it remains to prove the following lemma.
\begin{lemma}\label{lem-poincare}
For any $u\in C^{1}([u_+, u_-])$, the following inequality holds.
\[
\int_{u_{+}}^{u_{-}} (u-\bar{u})^2 dx \le\frac{5}{6} \int_{u_{+}}^{u_{-}}  (u_{-}- x)(x-u_{+})|u^{\prime}|^2 dx.
\] 
where $\bar{u}$ is the mean of $u$ over $[ u_{+},  u_{-}]$.
\end{lemma}
\begin{proof}
Let $v:=u-\bar{u}$. We start with the fundamental theorem of calculus:
\[
v(x) = v(y) +   \int _y^x v^{\prime}(z) dz,
\]
Since $v$ has mean zero, integrating this equality in $y$, we have 
\beq\label{v-1}
v(x)=\frac{1}{\alpha}\int_{u_{+}}^{u_{-}} \int _y^x v^{\prime}(z) dzdy.
\eeq
To compute the $L^2$-norm of $v$, we use the indicator function $\chi_{I(a,b)}$ defined on the interval $I(a,b): = [\min\{a,b\}, \max\{a,b\}]$, i.e.,
\begin{align*}
\begin{aligned} 
\chi_{I(a,b)}=\left\{ \begin{array}{ll}
          \chi_{[a,b]} & \mbox{if $a\le b$},\\
         \chi_{[b,a]} & \mbox{if $a>b$}.\end{array} \right.
\end{aligned}
\end{align*}
Then, we have from \eqref{v-1} that
\begin{align*}
\begin{aligned} 
 \int_{u_{+}}^{u_{-}} v^2 dx&= \frac{1}{\alpha^2}\int_{u_{+}}^{u_{-}} \Big| \int_{u_{+}}^{u_{-}} \int _y^x v^{\prime}(z) dzdy \Big|^2 dx\\
 &\le \frac{1}{\alpha^2}\int_{u_{+}}^{u_{-}} \Big| \int_{u_{+}}^{u_{-}} \int_{u_{+}}^{u_{-}} |v^{\prime}(z)| \chi_{I(x,y)}  dzdy \Big|^2 dx\\
 &\le \frac{1}{\alpha^2}\int_{u_{+}}^{u_{-}} \Big( \int_{u_{+}}^{u_{-}} \int_{u_{+}}^{u_{-}} \chi_{I(x,y)}  dzdy \Big)  \Big( \int_{u_{+}}^{u_{-}} \int_{u_{+}}^{u_{-}} |v^{\prime}(z)|^2 \chi_{I(x,y)}  dzdy \Big)dx,
 \end{aligned}
\end{align*}
where the last inequality is due to the Cauchy-Schwarz inequality.\\
Notice that it follows from the definition of $\chi_{I(a,b)}$ that for any integrable function $f$ and fixed $x\in [u_{+}, u_{-}]$,
\begin{align*}
\begin{aligned} 
\int_{u_{+}}^{u_{-}} \int_{u_{+}}^{u_{-}} \chi_{I(x,y)}(z) f(x,y,z)  dzdy &= \int_{u_{+}}^x \int_{u_{+}}^{u_{-}} \chi_{I(x,y)}f dzdy +\int_{x}^{u_{-}} \int_{u_{+}}^{u_{-}} \chi_{I(x,y)} f  dzdy \\
&=  \int_{u_{+}}^{x} \int _y^x f dzdy + \int_{x}^{u_{-}} \int _x^y f dzdy.
\end{aligned}
\end{align*}
Thus, applying the equality above with $f=1, | v^{\prime}(z)|^2$ twice, we have
\begin{align*}
\begin{aligned} 
 &\frac{1}{\alpha^2}\int_{u_{+}}^{u_{-}} \Big( \int_{u_{+}}^{u_{-}} \int_{u_{+}}^{u_{-}} \chi_{I(x,y)}  dzdy \Big)  \Big( \int_{u_{+}}^{u_{-}} \int_{u_{+}}^{u_{-}} |v^{\prime}(z)|^2 \chi_{I(x,y)}  dzdy \Big)dx \\
&\quad = \frac{1}{\alpha^2}\int_{u_{+}}^{u_{-}}  \Big(  \underbrace{\int_{u_{+}}^{x} \int _y^x 1 dzdy + \int_{x}^{u_{-}} \int _x^y 1 dzdy}_{I} \Big)\\
&\hspace{2cm} \times   \Big(  \underbrace{\int_{u_{+}}^{x} \int _y^x |v^{\prime}(z)|^2 dzdy +  \int_{x}^{u_{-}} \int _x^y | v^{\prime}(z)|^2 dzdy}_{II} \Big) dx.
\end{aligned}
\end{align*}
We use the Fubini's theorem to compute
\begin{align*}
\begin{aligned} 
II &=\int_{u_{+}}^{x}  \int_{u_{+}}^{z} |v^{\prime}(z)|^2 dydz +  \int_{x}^{u_{-}}  \int_{z}^{u_{-}}  | v^{\prime}(z)|^2 dydz\\
&= \int_{u_{+}}^{x}  (z-u_{+}) |v^{\prime}(z)|^2 dz + \int_{x}^{u_{-}} (u_{-}-z) | v^{\prime}(z)|^2 dz.
\end{aligned}
\end{align*}
Since 
\[
I= \frac{(x-u_{+})^2}{2} +\frac{(x-u_{-})^2}{2},
\]
we have
\begin{align*}
\begin{aligned} 
\int_{u_{+}}^{u_{-}} v^2 dx&\le \frac{1}{2\alpha^2} \int_{u_{+}}^{u_{-}}\int_{u_{+}}^{x} \Big( {(x-u_{+})^2} +(x-u_{-})^2 \Big) (z-u_{+}) |v^{\prime}(z)|^2 dz dx \\
&\hspace{2cm}  + \frac{1}{2\alpha^2} \int_{u_{+}}^{u_{-}}\int_{x}^{u_{-}}\Big( (x-u_{+})^2 +(x-u_{-})^2 \Big) (u_{-}-z) | v^{\prime}(z)|^2 dzdx \\
&=: \mathcal{I}_1 + \mathcal{I}_2.
\end{aligned}
\end{align*}
Using the Fubini's theorem again, we have
\begin{align}
\begin{aligned}\label{i-1} 
 \mathcal{I}_1=\frac{1}{2\alpha^2} \int_{u_{+}}^{u_{-}}   \Big(  \underbrace{\int_{z}^{u_{-}} ( (x-u_{+})^2+(x-u_{-})^2 ) dx}_{J} \Big) (z-u_{+}) |v^{\prime}(z)|^2  dz
\end{aligned}
\end{align}
Since the simple computation yields
\begin{align*}
\begin{aligned} 
J &= \int_{z}^{u_{-}} (2x^2 -2(u_++u_-)x + (u_+^2 +u_-^2) )dx\\
&=\frac{1}{3} (u_- - z)\Big( 2(u_-^2 +u_- z + z^2) -3(u_+ +u_-)(u_- + z) + 3(u_+^2 +u_-^2) \Big)\\
&=\frac{1}{3} (u_- - z)(2u_-^2 + 3u_+^2 + 2z^2 -3u_-u_+ -3u_+z -u_-z), 
\end{aligned}
\end{align*}
we have
\begin{align*}
\begin{aligned} 
 \mathcal{I}_1=\frac{1}{6\alpha^2} \int_{u_{+}}^{u_{-}}  (2u_-^2 + 3u_+^2 + 2z^2 -3u_-u_+ -3u_+z -u_-z)(u_- - z)(z-u_{+}) |v^{\prime}(z)|^2  dz.
\end{aligned}
\end{align*}
By using symmetry of $ \mathcal{I}_1$ and $ \mathcal{I}_2$, we can easily get 
\begin{align*}
\begin{aligned} 
 \mathcal{I}_2=\frac{1}{6\alpha^2}\int_{u_{+}}^{u_{-}} (3u_-^2 + 2u_+^2 + 2z^2 -3u_-u_+ -u_+z -3u_-z)(u_- - z)(z-u_{+})  | v^{\prime}(z)|^2 dz .
\end{aligned}
\end{align*}
Indeed, if we use the notation
\[
I(a,b):= \int_{z}^{b} ( (x-a)^2+(x-b)^2 )  (z-a) |v^{\prime}(z)|^2 dx,
\]
the $\mathcal{I}_1$ in \eqref{i-1} can be written as $\mathcal{I}_1=\frac{1}{2\alpha^2} \int_{u_{+}}^{u_{-}} I(u_+, u_-) dz$. Then we use the Fubini's theorem to get
\begin{align*}
\begin{aligned} 
 \mathcal{I}_2&=\frac{1}{2\alpha^2} \int_{u_{+}}^{u_{-}}  \int_{u_{+}}^z ( (x-u_{+})^2+(x-u_{-})^2 )  (u_{-}-z) |v^{\prime}(z)|^2 dx dz\\
 &=\frac{1}{2\alpha^2} \int_{u_{+}}^{u_{-}} I( u_-,u_+) dz.
\end{aligned}
\end{align*}
Finally, we combine $ \mathcal{I}_1$ with $ \mathcal{I}_2$ above to have
\begin{align*}
\begin{aligned} 
 \int_{u_{+}}^{u_{-}} v^2 dx &=\frac{1}{6\alpha^2} \int_{u_{+}}^{u_{-}} \Big( 5(u_- -u_+)^2 +4(z- u_+)(z-u_-) \Big)  (u_{-}- z)(z-u_{+})|v^{\prime}(z)|^2 dz \\
&\le\frac{5}{6} \int_{u_{+}}^{u_{-}} (u_{-}- z)(z-u_{+})|v^{\prime}(z)|^2 dz - \frac{4}{6\alpha^2}\int_{u_{+}}^{u_{-}} \Big( (u_{-}- z)(z-u_{+}) \Big)^2 |v^{\prime}(z)|^2 dz\\
&\le\frac{5}{6} \int_{u_{+}}^{u_{-}} (u_{-}- z)(z-u_{+})|v^{\prime}(z)|^2 dz.
\end{aligned}
\end{align*}
\end{proof}

\begin{remark}
This kind of inequality has been handled in a more general setting \cite{C-W}, but does not provide a generic constant concretely as $\frac{5}{6}(<1)$ in our inequality. 

\end{remark}

\subsection{Convergence toward viscous shock} In this part, we derive the decay estimate \eqref{convergence}. First of all, since $\lambda>0$ in \eqref{d-2} by the assumption $\|g^{\prime\prime}\|_{L^{\infty}(\bbr)}< \frac{2}{11}a$, we have a positive dissipation as
\beq\label{1-1}
\frac{d}{dt} \int_{-\infty}^{\infty} |V-S_1|^2 dx \le -\lambda  \int_{u_{+}}^{u_{-}}  (u_{-}- y)(y-u_{+})|\partial_y w|^2 dy.
\eeq
We see that \eqref{layer} and \eqref{quad} yield that
\begin{align*}
\begin{aligned}
&a \int_{u_{+}}^{u_{-}}  (u_{-}- y)(y-u_{+})|\partial_y w|^2 dy \\
&\quad= -\int_{u_{+}}^{u_{-}} \Big( A(y) -A(u_{-}) -\sigma (y -u_{-}) \Big) |\partial_y w|^2 dy\\
&\qquad+\int_{u_{+}}^{u_{-}} \Big( g(y) -g(u_{-}) -\frac{g(u_-) - g(u_+)}{u_- -u_+} (y -u_{-}) \Big) |\partial_y w|^2 dy.
\end{aligned}
\end{align*}
Moreover, using the change of variable \eqref{w} and \eqref{I_4}, we have
\begin{align*}
\begin{aligned}
(a+\frac{1}{2}\|g''\|_{L^{\infty}(\bbr)}) \int_{u_{+}}^{u_{-}}  (u_{-}- y)(y-u_{+})|\partial_y w|^2 dy \ge \int_{-\infty}^{\infty}|\partial_x(V-S_1)|^2 dx,
\end{aligned}
\end{align*}
which together with \eqref{1-1} implies that
\beq\label{ineq-1}
\frac{d}{dt} \int_{-\infty}^{\infty} |V-S_1|^2 dx \le -\alpha \int_{-\infty}^{\infty}|\partial_x(V-S_1)|^2 dx,
\eeq
where 
\[
\alpha:=\frac{\lambda}{a+\frac{1}{2}\|g''\|_{L^{\infty}(\bbr)}} >0.
\]
To get the decay estimate, we will use the Gagliardo-Nirenberg interpolation inequality:
\beq\label{GN}
\|V-S_1\|_{L^2(\bbr)} \le C \|V-S_1\|_{L^1(\bbr)}^{2/3} \|\partial_x(V-S_1)\|_{L^2(\bbr)}^{1/3},
\eeq


To this end, we first control $\|V-S_1\|_{L^1(\bbr)}$ as follows. 

\subsubsection{$L^1$-uniform bound of $V-S_1$} We here use the Lemma \ref{lem-k} below for the $L^1$-contraction result to get the $L^1$-estimate of $V-S_1$.\\
For that, we decompose $V-S_1$ as a sum of two parts:
\begin{align}
\begin{aligned}\label{com}
V(t,x) - S_1(x) &= \Big(V(t,x) - S_1(x - \sigma t +X(t)) \Big) + \Big(S_1(x- \sigma t+X(t)) - S_1(x)\Big)\\
& =: w_1 + w_2,
\end{aligned}
\end{align}
where $\sigma$ is the velocity of $S_1$ and $X$ is the shift satisfying \eqref{curve}.\\
To show $L^1$-uniform bound of $w_1$, we use the following lemma on the $L^1$-contraction for solutions to the scalar viscous conservation laws. We refer to \cite{Serre} for its proof based on Kruzkhov entropy pair (see also \cite{K}).
\begin{lemma}\label{lem-k}
Let $u$ and $v$ be solutions to \eqref{main} with Lipschitzian flux $A$. If the initial data $u_0$, $v_0$ satisfy $u_0-v_0\in L^1(\bbr)$, then the following $L^1$-stability holds:
\begin{align}
\begin{aligned}\label{L1}
\|u-v\|_{L^{1}(\bbr)} \le \|u_0-v_0\|_{L^{1}(\bbr)},\quad t>0.
\end{aligned}
\end{align}
\end{lemma}

Applying \eqref{L1} to our solutions $U(t,x)$ and $S_1(x-\sigma t)$ of the Burgers equation \eqref{main}, we get 
\[
\|U-S_1(\cdot - \sigma t)\|_{L^{1}(\bbr)} \le \|U_0-S_1\|_{L^{1}(\bbr)}.
\] 
Since $U(t,x) = V(t, x-X(t))$, we use the assumption $ U_0-S_1\in L^{1}$ to have
\begin{align}
\begin{aligned}\label{w1}
\|w_1\|_{L^{1}} &= \|V(t, \cdot -X(t)) - S_1(\cdot- \sigma t)\|_{L^1}\\
&= \|U(t, \cdot) - S_1(\cdot- \sigma t)\|_{L^1}\\
&\le \|U_0-S_1\|_{L^{1}}.
\end{aligned}
\end{align}
If we denote 
$$
\tau(t)=\sigma t-X(t),
$$
we have 
$$
w_2(t,x)=S_1(x-\tau(t))-S_1(x).
$$
Since $S_1$ is decreasing,   $w_2(t,x)$ has the same sign as $ \tau(t)$ and 
\begin{align*}
\begin{aligned}
\int_{\bbr} |w_2| dx &= \sgn{(\tau(t))}\int_{\bbr} \Big(S_1(x- \tau(t)) - S_1(x) \Big) dx \\
& = \sgn{(\tau(t))} \int_{\bbr} \int_0^{-\tau(t)} \partial_y S_1(x+y) dydx\\
& = - \int_{\bbr} \int_0^{|\tau(t)|} \partial_y S_1(x+y) dydx.
\end{aligned}
\end{align*}
Then, we use the Fubini's theorem and $S_1(\pm\infty)=u_{\pm}$ to get
\[
\int_{\bbr} |w_2| dx =  |\tau(t)|(u_- - u_+)= |\sigma t-X(t)|(u_- - u_+).
\]
We now need to show the $L^{\infty}$-bound of $X(t)-\sigma t$ to get the $L^1$-uniform bound of $w_2$. 

\subsubsection{$L^{\infty}$-bound of $X(t)-\sigma t$.} We start with 
$$
\|w_2(t)\|^2_{L^2}=\int_{\R}|S_1(x-\tau)-S_1(x)|^2\,dx=: F(\tau),
$$
for $\tau=\sigma t-X(t)$.
The function $F$, as function of the variable $\tau$, is even (as it can be proven by the change of variable $y=x-\tau$ in the integral). For $\tau>0$, we have 
\begin{eqnarray*}
\frac{\partial F}{\partial \tau}(\tau)&=& 2\int_{\R}[S_1(x-\tau)-S_1(x)] \left(-\frac{\partial S_1}{\partial x}\right)(x-\tau)\,dx\\
&=& 2\int_{\R}\int_{x-\tau}^x\left(-\frac{\partial S_1}{\partial y}\right)(y)\,dy\left(-\frac{\partial S_1}{\partial x}\right)(x-\tau)\,dx\\
&=& 2\int_{\R}\int_{x}^{x+\tau}\left(- S'_1\right)(y)\left(- S'_1\right)(x)\,dy\,dx>0,
\end{eqnarray*}
which is positive since $(-S'_1)$ is positive. Moreover for $\tau>1$, we have
$$
\frac{\partial F}{\partial \tau}(\tau)\geq2\int_{\R}\int_{x}^{x+1}(-S'_1)(y) (-S'_1)(x)\,dy\,dx=\beta>0.
$$
Hence, for $\tau>1$
$$
F(\tau)\geq F(1)+\beta (\tau-1)\geq \beta(\tau-1),
$$
and 
$$
|\tau|\leq \frac{F(\tau)}{\beta}+1,
$$
which is still true for $\tau\leq 1$, since this is obvious for $\tau\in (-1,1)$, and $F$ is even.\\
For $|\tau|=|\sigma t-X(t)|$, this gives
\[
|X(t)- \sigma t|\le \frac{1}{\beta}\int_{\bbr} |w_2(t,x)|^2 dx+1.
\]
We now use
\[
|V-S_1|^2 = (w_1+w_2)^2 \ge w_2^2 - 2|w_1 w_2|,
\]
and \eqref{contract}, \eqref{w1} and $\|S_1\|_{L^{\infty}}= u_- -u_+$ to get
\begin{align*}
\begin{aligned}
|X(t)- \sigma t|&\le \frac{1}{\beta} \int_{\bbr} |w_2(t,x)|^2 dx+1\\
&\le\frac{1}{\beta}( \| V-S_1 \|_{L^2}^2 +  2\| w_1 w_2 \|_{L^1})+1 \\
& \le\frac{1}{\beta}( \| U_0-S_1 \|_{L^2}^2 + 2\|w_2\|_{L^{\infty}} \| w_1\|_{L^1})+1\\
&\le \frac{1}{\beta}( \| U_0-S_1 \|_{L^2}^2 +  4(u_- -u_+)\| U_0-S_1 \|_{L^1})+1.
\end{aligned}
\end{align*}
Therefore, for all $t$, we have
\[
|X(t)- \sigma t |\le C (1 +  \| U_0-S_1 \|_{L^2}^2 +  \| U_0-S_1 \|_{L^1}),
\]
where $C>0$ is a generic constant only depending on $u_-$, $u_+$ and the flux $A$.\\ 
Hence we have from \eqref{com} and estimates above that
\[
\|V-S_1\|_{L^1} \le C (1 +  \| U_0-S_1 \|_{L^2}^2 +  \| U_0-S_1 \|_{L^1}).
\]
For convenience, we put 
\beq\label{cons}
C_0:=C (1 +  \| U_0-S_1 \|_{L^2}^2 +  \| U_0-S_1 \|_{L^1}).
\eeq
\\

We now get from \eqref{GN} that 
\[
\|V-S_1\|_{L^2}^3 \le C\|V-S_1\|_{L^1}^2\|\partial_x (V-S_1)\|_{L^2} \le  C_0^2\|\partial_x (V-S_1)\|_{L^2}.
\]
Thus it follows from \eqref{ineq-1} that
\begin{align*}
\begin{aligned} 
\frac{d}{dt} \|V-S_1\|_{L^2}^2  \le -\alpha\|\partial_x (V-S_1)\|_{L^2}^2 \le -\frac{\alpha}{C_0^{4}} \|V-S_1\|_{L^2}^6.
\end{aligned}
\end{align*}
This inequality implies the decay estimate 
\begin{align*}
\begin{aligned} 
\|V-S_1\|_{L^{2}}^4 &\le  \frac{C_0^{4}\|U_0 - S_1\|_{L^2}^4}{C_0^{4}+ t \|U_0 - S_1\|_{L^2}^4 }.
\end{aligned}
\end{align*}
Using the inequality $2(\alpha+\beta)^{1/4} \ge \alpha^{1/4} + \beta^{1/4}$, we have
\begin{align*}
\begin{aligned} 
\|V-S_1\|_{L^{2}} &\le \Big( \frac{C_0^{4}\|U_0 - S_1\|_{L^2}^4}{C_0^{4}+ t \|U_0 - S_1\|_{L^2}^4 } \Big)^{1/4}\\
&\le \frac{2C_0\|U_0 - S_1\|_{L^2}}{C_0+t^{1/4}\|U_0 - S_1\|_{L^2} }, \quad t>0,
\end{aligned}
\end{align*}
which completes the decay estimate \eqref{convergence}.

\section{Proof of Theorem \ref{thm-example}} 
In this section, we construct a strictly convex flux $A$ and initial data $U_0$ as a small perturbation of viscous shock $S_1$ in order to make the dissipation $D$ to be negative for very short time, which definitely complete the proof. Without loss of generality, we only consider the simple case when two endpoints are given by $u_{+}=-a$, $u_{-}=a$ for given $a>0$. Since we may construct the convex flux $A$ satisfying $A(-a)=A(a)=0$ below, the shock speed $\sigma=0$. Thus the associated dissipation $D$ in \eqref{dissipate} becomes
\begin{align}
\begin{aligned}\label{dissipate-2}
D(t) &= 2 \dot{X}(t) \int_{-a}^{a} w dy -2  \int_{-a}^{a} A(w+y|y) dy -2  \int_{-a}^{a} A(y)  |\partial_y w|^2 dy,
\end{aligned}
\end{align}
\subsection{small perturbation of $S_1$} 
For a given $\eps>0$, we consider an initial $\eps$-perturbation $w (0,y)$ of $S_1$, that is, we replace $w (0,y)$ by $\eps \phi (y)$, where $\phi$ is a function of order $\mathcal{O}(1)$. Doing the Taylor expansion of $A$, the relative flux $A(\eps\phi+y|y)$ in \eqref{dissipate-2} can be written as
\begin{align*}
\begin{aligned}
A(\eps\phi+y|y)&=A(\eps\phi+y)-A(y)-A^{\prime}(y)\eps \phi \\
&=\frac{1}{2}A^{\prime\prime}(y) \eps^2 \phi^2 +\mathcal{O}(\eps^3). 
\end{aligned}
\end{align*}
Thus, under the $\eps$-perturbation framework, the initial dissipation in \eqref{dissipate-2} becomes
\begin{align}
\begin{aligned}\label{perturb}
D(0) &= 2 \dot{X}(0) \int_{-a}^{a} \eps\phi dy - 2 \int_{-a}^{a} A(\eps \phi+y|y) dy  -2\int_{-a}^{a} A(y)  | \eps \phi^{\prime} |^2 dy\\
&=2  \eps\dot{X}(0) \int_{-a}^{a}\phi dy -\eps^2 \Big(\int_{-a}^{a} A^{\prime\prime}(y) |\phi|^2 dy + 2  \int_{-a}^{a} A(y)  |\phi^{\prime} |^2 dy + \mathcal{O}(\eps) \Big).
\end{aligned}
\end{align}
\subsection{construction of $A$ and $U_0$} 
For given $\alpha\in (0,a)$, we first define two continuous functions $\bar{A}_{\alpha}$ and $\psi_{\alpha}$ by
\begin{align*}
\begin{aligned}
\bar{A}_{\alpha}(x)=\left\{ \begin{array}{ll}
         -a -x & \mbox{if $-a < x < -a+\alpha$},\\
        -\alpha & \mbox{if $-a+\alpha \le x < a-\alpha$},\\
         x -a & \mbox{if $a-\alpha \le x < a$},\end{array} \right.
\end{aligned}
\end{align*}
\begin{align*}
\begin{aligned}
\psi_{\alpha}(x)=\left\{ \begin{array}{ll}
         -\sqrt{x+a} & \mbox{if $-a < x < -a+\alpha$},\\
        \frac{\sqrt{\alpha}}{a-\alpha} x & \mbox{if $-a+\alpha \le x < a-\alpha$},\\
         \sqrt{a-x} & \mbox{if $a-\alpha \le x < a$},\end{array} \right.
\end{aligned}
\end{align*}
First of all, we compute formally
\begin{align*}
\begin{aligned}
\int_{-a}^{a} \bar{A}_{\alpha}^{\prime\prime}(y) |\psi_{\alpha}(y)|^2 dy &= \int_{-a}^{a} (\delta_{-a+\alpha} +\delta_{a-\alpha})  |\psi_{\alpha}|^2 dy\\
&= |\psi_{\alpha}(-a+\alpha)|^2 + |\psi_{\alpha}(a-\alpha)|^2 = 2\alpha, \\
 \int_{-a}^{a} \bar{A}_{\alpha}(y)  |\psi_{\alpha}^{\prime}(y) |^2 dy &=  -2 \Big[ \int_{0}^{a-\alpha} \alpha (\frac{\sqrt{\alpha}}{a-\alpha})^2 d y+  \int_{a-\alpha}^a (y-a)(\frac{-1}{2\sqrt{a-y}})^2 dy \Big] \\
 &= -\frac{2\alpha^2}{a-\alpha}-\frac{\alpha}{2}.
\end{aligned}
\end{align*}
Then, we choose $\alpha_*<\frac{a}{5}$ small enough so that
\beq\label{formal}
\int_{-a}^{a} \bar{A}_{\alpha_*}^{\prime\prime}(y) |\psi_{\alpha_*}|^2 dy + 2  \int_{-a}^{a} \bar{A}_{\alpha_*}(y)  |\psi_{\alpha_*}^{\prime} |^2 dy >0.
\eeq
Since the inequality \eqref{formal} is strict, we can consider the smooth approximations of $\bar{A}_{\alpha_*}$ and $\psi_{\alpha_*}$, for which the inequality \eqref{formal} is still true by rigorous computation. More precisely, by using the Gaussian mollifier, there exists the smooth approximations $A$ and $\phi$ of $\bar{A}_{\alpha_*}$ and $\psi_{\alpha_*}$ respectively, such that
\[
A^{\prime\prime}>0,\quad \int_{-a}^{a} \phi dy= \int_{-a}^{a}\psi_{\alpha_*} dy= 0,
\]
and the inequality \eqref{formal} still holds as
\[
\int_{-a}^{a} A^{\prime\prime}(y) |\phi|^2 dy + 2  \int_{-a}^{a} A(y)  |\phi^{\prime} |^2 dy >0.
\]
We can still choose sufficiently small $\eps_0>0$ such that
\beq\label{ss}
\int_{-a}^{a} A^{\prime\prime}(y) |\phi|^2 dy + 2  \int_{-a}^{a} A(y)  |\phi^{\prime} |^2 dy+ \mathcal{O}(\eps_0) >0.
\eeq
Since $\int_{-a}^{a} \phi dy=0$, it follows from \eqref{perturb} and \eqref{ss} that
\beq\label{diss-2}
D(0) = -\eps_0^2 \Big(\int_{-a}^{a} A^{\prime\prime}(y) |\phi|^2 dy + 2  \int_{-a}^{a} A(y)  |\phi^{\prime} |^2 dy + \mathcal{O}(\eps_0) \Big) <0.
\eeq
If we consider a initial data $U_0$ constructed by $U_0(x)=\eps_0\phi(S_1(x)) + S_1(x)$, then we have
\[
U_0(x)-S_1(x)=\eps_0\phi(S_1(x)) =\eps_0 \phi (y) =w(0,y),
\]
which implies that $D(0)<0$ from \eqref{diss-2} for the flux $A$ and the initial data $U_0$.\\ 
Since $D(t)$ is smooth for $t>0$, for any Lipschitz function $X(t)$, there exists a small time $T^*$ depending on the Lipschitz constant of $X(t)$ such that $D(t)<0$, for $0\le t < T^*$. Hence we conclude the proof.

\section{Proof of Theorem \ref{thm-improve}} 
In this section, we prove the Theorem \ref{thm-improve}. We begin by recalling the inviscid problem
\begin{align}
\begin{aligned} \label{inviscid main-2}
&\partial_t U^{\eps} + \partial_x A(U^{\eps}) = \eps\partial^2_{xx} U^{\eps}, \quad t>0,~x\in \bbr,\\
&U^{\eps}(0,x) = U_0(x).
\end{aligned}
\end{align}
We here present two kinds of improvements. The first improvement \eqref{im-1} is based on the contraction \eqref{contract} and the second improvement \eqref{im-2} is related to the decay estimate \eqref{convergence}.
 
\subsection{Improvement based on the contraction}
For a solution $U^{\eps}$ to \eqref{inviscid main-2}, we consider 
\[
U(t,x):=U^{\eps}(\eps t, \eps x),
\]
then $U$ is a solution to 
\begin{align*}
\begin{aligned}
&\partial_t U + \partial_x A(U) = \partial^2_{xx} U, \quad t>0,~x\in \bbr,\\
&U(0,x) = U_0(\eps x).
\end{aligned}
\end{align*}  
We now use the contraction property \eqref{contract} to get
\begin{align*}
\begin{aligned}
\|U(t,\cdot)-S_1(\cdot-X(t))\|_{L^{2}(\bbr)} \le \|U_0-S_1\|_{L^{2}(\bbr)},\quad t>0,
\end{aligned}
\end{align*}  
where the shift $X(t)$ verifies \eqref{curve}.\\
Then, by rescaling $ t \rightarrow \frac{t}{\eps}$ and $x \rightarrow \frac{x}{\eps}$, i.e.,
\beq\label{relation-1}
\int_{\bbr} |U(t,x)-S_1(x-X(t))|^2 dx = \frac{1}{\eps}\int_{\bbr} \Big|U^{\eps}(t,x) - S_1(\frac{x-Y(t)}{\eps})\Big|^2 dx,
\eeq
where the shift $Y$ is defined by $Y(t)=\eps X(t/\eps)$, we get 
\[
\Big\|U^{\eps}(t,\cdot)-S_1(\frac{\cdot-Y(t)}{\eps})\Big\|_{L^{2}(\bbr)} \le \Big\|U_0-S_1(\frac{\cdot}{\eps})\Big\|_{L^{2}(\bbr)}.
\]
Therefore we have
\begin{align*}
\begin{aligned}
&\|U^{\eps}(t,\cdot)-S_0(\cdot-Y(t))\|_{L^{2}(\bbr)} \\
&\qquad \le \Big\|U^{\eps}(t,\cdot)-S_1(\frac{\cdot-Y(t)}{\eps})\Big\|_{L^{2}(\bbr)} + 
\Big\|S_1(\frac{\cdot-Y(t)}{\eps})-S_0(\cdot-Y(t))\Big\|_{L^{2}(\bbr)}\\
&\qquad \le \Big\|U_0-S_1(\frac{\cdot}{\eps})\Big\|_{L^{2}(\bbr)}  + C\sqrt{\eps}\\
&\qquad \le \Big\|U_0-S_0\Big\|_{L^{2}(\bbr)} + \Big\|S_0-S_1(\frac{\cdot}{\eps})\Big\|_{L^{2}(\bbr)}  + C\sqrt{\eps}\\
&\qquad \le \Big\|U_0-S_0\Big\|_{L^{2}(\bbr)} + C\sqrt{\eps}\\
\end{aligned}
\end{align*} 
where we have used the fact that for any function $\beta$ of $t$,
\[
\Big\|S_1(\frac{\cdot-\beta(t)}{\eps})-S_0(\cdot-\beta(t))\Big\|_{L^{2}(\bbr)} = \sqrt{\eps} \|S_1-S_0 \|_{L^{2}(\bbr)}.
\]

\subsection{Improvement based on the decay estimate} 
For the other improvement, we use the decay estimate \eqref{convergence} to get
\begin{align*}
\begin{aligned}
&\|U(t,\cdot)-S_1(\cdot-X(t))\|_{L^{2}}\\
&\qquad \le \frac{C\|U(0,\cdot) - S_1\|_{L^2}(1+  \|U(0,\cdot)-S_1\|_{L^{1}}+ \|U(0,\cdot)-S_1\|_{L^{2}}^2)}{1+ t^{1/4} \|U(0,\cdot) - S_1\|_{L^2}}.
\end{aligned}
\end{align*} 
Then, by the rescaling \eqref{relation-1}, we get
\begin{align*}
\begin{aligned}
&\|U^{\eps}(t,\cdot) - S_1(\frac{\cdot-Y(t)}{\eps})\|_{L^{2}} \\
&\qquad\le \frac{C\eps^{3/4}\|U_0 - S_1(\frac{\cdot}{\eps})\|_{L^2}(1+  \eps^{-1}(\|U_0-S_1(\frac{\cdot}{\eps})\|_{L^{1}}+ \|U_0-S_1(\frac{\cdot}{\eps})\|_{L^{2}}^2))}{\eps^{3/4}+ t^{1/4} \|U_0 - S_1(\frac{\cdot}{\eps})\|_{L^2}}.
\end{aligned}
\end{align*} 
If we consider the small initial perturbation as 
\beq\label{small-per}
\|U_0-S_0\|_{L^{1}}+ \|U_0-S_0\|_{L^{2}}^2\le C_1\eps,
\eeq
 we have
\begin{align*}
\begin{aligned}
&\|U_0-S_1(\frac{\cdot}{\eps})\|_{L^{1}} \le \|U_0-S_0\|_{L^{1}} + \|S_0-S_1(\frac{\cdot}{\eps})\|_{L^{1}} \le C\eps, \\
&\|U_0-S_1(\frac{\cdot}{\eps})\|_{L^{2}}^2 \le \|U_0-S_0\|_{L^{2}}^2 + \|S_0-S_1(\frac{\cdot}{\eps})\|_{L^{2}}^2 \le C\eps,
\end{aligned}
\end{align*} 
which yields
\begin{align*}
\begin{aligned}
\|U^{\eps}(t,\cdot) - S_1(\frac{\cdot-Y(t)}{\eps})\|_{L^{2}}^2 &\le \frac{C\eps^{3/2}\|U_0 - S_1(\frac{\cdot}{\eps})\|_{L^2}^2}{\eps^{3/2}+ t^{1/2} \|U_0 - S_1(\frac{\cdot}{\eps})\|_{L^2}^2}\\
&\le \frac{C\eps^{5/2}}{\eps^{3/2}+ t^{1/2} (\|S_1(\frac{\cdot}{\eps})-S_0\|_{L^{2}}^2-\|U_0 - S_0\|_{L^2}^2)}.
\end{aligned}
\end{align*} 
If we consider some constant $C_1$ in \eqref{small-per} such that 
\[
C_1< \|S_1-S_0 \|_{L^{2}}^2=\frac{1}{\eps}\|S_1(\frac{\cdot}{\eps})-S_0\|_{L^{2}}^2,
\]
then we have
\begin{align*}
\begin{aligned}
\|U^{\eps}(t,\cdot) - S_1(\frac{\cdot-Y(t)}{\eps})\|_{L^{2}}^2 
&\le \frac{C\eps^{5/2}}{\eps^{3/2}+ t^{1/2} ( \|S_1-S_0 \|_{L^{2}}^2- C_1)\eps}\\
&\le \frac{C\eps^{3/2}}{\eps^{1/2}+ t^{1/2} }.
\end{aligned}
\end{align*}

\end{document}